\documentclass[11pt]{amsart}
\usepackage{fullpage}
\usepackage{color}
\usepackage{pstricks,pst-node,pst-plot} 
\usepackage{graphicx,psfrag} 
\usepackage{color} 
\usepackage{tikz}
\usepackage{pgffor}
\usepackage{hyperref}
\usepackage{todonotes}
\usepackage{subfigure}
\usepackage{mathtools}

\usepackage{bm}
\usepackage{multirow}

\usepackage{perpage}	 
	
\MakePerPage{footnote}	 

\newtheorem{theorem}{Theorem}[section]
\newtheorem*{theorem-non}{Theorem}
\newtheorem{lemma}[theorem]{Lemma}

\newtheorem{proposition}[theorem]{Proposition}

\newtheorem{conjecture}[theorem]{Conjecture}

\theoremstyle{definition}
\newtheorem{definition}[theorem]{Definition}
\newtheorem{example}[theorem]{Example}

\theoremstyle{remark}

\numberwithin{equation}{section}

\definecolor{gray}{rgb}{.5,.5,.5}

\definecolor{black}{rgb}{0,0,0}

\definecolor{blue}{rgb}{0,0,1}

\definecolor{red}{rgb}{1,0,0}

\definecolor{green}{rgb}{0,1,0}

\definecolor{yellow}{rgb}{1,1,.4}

\newrgbcolor{purple}{.5 0 .5}
\newrgbcolor{black}{0 0 0}
\newrgbcolor{white}{1 1 1}
\newrgbcolor{gold}{.5 .5 .2}
\newrgbcolor{darkgreen}{0 .5 0}
\newrgbcolor{gray}{.5 .5 .5}
\newrgbcolor{lightgray}{.75 .75 .75}
\newrgbcolor{lightred}{.75 0 0}

\begin{document}

\title{A central limit theorem\\ for the signatures of 2-bridge knots}

\author{Moshe Cohen}
\address{Mathematics Department, State University of New York at New Paltz, New Paltz, NY 12561}
\email{cohenm@newpaltz.edu}

\author{Cody Baker}
\address{Mathematics Department, State University of New York at New Paltz, New Paltz, NY 12561}
\email{bakerc16@newpaltz.edu}

\author{Henry Dam}
\address{Mathematics Department, State University of New York at New Paltz, New Paltz, NY 12561}
\email{damh1@newpaltz.edu}

\author{Rebecca Felber}
\address{Mathematics Department, State University of New York at New Paltz, New Paltz, NY 12561}
\email{felberr1@newpaltz.edu}

\author{Neal Madras}
\address{Department of Mathematics and Statistics, York University, Toronto, ON M3J 1P3}
\email{madras@yorku.ca}

\author{Ritvik Saha}
\address{15 Alan Crest Dr, Hicksville, NY 11801}
\email{ritvik.saha24@gmail.com}

\author{Daisy Thackrah}
\address{Mathematics Department, State University of New York at New Paltz, New Paltz, NY 12561}
\email{thackrad1@newpaltz.edu}

\begin{abstract}
Cohen, Lowrance, Madras, and Raanes \cite{CLMR} computed the average (absolute value of) signature over all 2-bridge knots with crossing number $c$ by introducing the number $s(c,\sigma)$ of 2-bridge knots of crossing number $c$ and signature $\sigma$.  Here we provide a closed formula for this number.  We use these calculations to show that the distribution of the signatures of 2-bridge knots with crossing number $c$ approaches a normal distribution as $c$ tends to infinity. 
\end{abstract}


\maketitle

\section{Introduction}

Many results from random topology show central limit theorems for genus distributions:  Chmutov and Pittel \cite{ChmPit:genus} for random chord diagrams; Chmutov and Pittel \cite{ChmPit:poly} for randomly gluing of the sides of several polygonal discs;  Even-Zohar and Farber \cite{EZFar} for surfaces as above but with boundary (a 2D version); and Shrestha  \cite{Shr:genus} for random square-tiled surfaces.   Unlike these random topological objects, knots have been carefully identified and tabulated \cite{knotinfo}.  A concern for generating (and proving results on) large random knots is whether these samples are representative.  The goal of the work below is to understand the asymptotics of a well-understood, very large infinite family of knots so that one could eventually compare with random knot models.

Many random knot models appear in the literature:  random walks in the braid group \cite{SorVil, DunTio}; random petal diagrams \cite{EZHLN, EZHLN:distpetal}; random grid diagrams \cite{Doig}; random meander diagrams \cite{OwaTsv}; randomly embedded polygons \cite{CanSho, CCRS, TriCNSS, CanSchSho}; random diagrams \cite{CanChaMas}; and others.  Using work by Koseleff and Pecker \cite{KosPecRou:10crossings, KosPec4, KosPec3, KosPec:Alex}, Cohen and Krishnan \cite{CoKr} and with Even-Zohar \cite{CoEZKr} developed and studied a model for random 2-bridge knots.  These knots arise in models for polymers in work of Beaton, Eng, and Soteros with others \cite{BES, BES:stats, BES+:nano, BES+:4}, where this study of randomness may be important.  See also Blair, Pongtanapaisan, and Soteros \cite{BPS}. Furthermore, insights by Cohen, Krishnan, and Even-Zohar on these knots led to discoveries which underpin the present work.

A knot diagram is a planar projection of a knot where the double-points are decorated by over- and under-crossing information. The crossing number $c(K)$ of a knot $K$ is the minimum number of such points over all possible diagrams of $K$.  The bridge number of a knot is the minimum number of local maxima  of a diagram of the knot over all possible diagrams.  For more background see the textbook by Cromwell \cite{Crom}.  The only 1-bridge knot is the unknot; a simple closed curve always has at least one local max.  The family of 2-bridge knots have two local (non-nested) maxima from which four strands descend; the second and third strands are twisted $a_1$ times (with sign indicating which direction they are twisted), the third and fourth strands are twisted $-a_2$ times, the second and third strands are twisted $a_3$ times, and so on, with some finite number $k$ of twist regions, after which there are two local minima.  See Figure \ref{fig:knot}.  One can write the finite continued fraction $[a_1,\ldots,a_k]=p/q$ for $a_i\in\mathbb{Z}$ to obtain a rational number, as in Kauffman and Lambropoulou \cite{KauffLamb:tangles, KauffLamb:class}, or use the Farey graph as in Hoste, Shanahan, and Van Cott \cite{HSVC}.

\begin{figure}[h!]
\begin{tikzpicture}[scale=.5]
\draw[thick] (0,-1) -- (6,-1);
\draw[thick] (0,0) arc (90:270:.5);
\draw[thick] (6,-1) arc (-90:90:1.5);
\draw[thick] (0,2) arc (90:270:.5);
\draw[thick] (6,0) arc (-90:90:.5);
\draw[thick] (0,2) -- (1,2);
\draw[thick] (1,0) -- (2,0);
\draw[thick] (2,2) -- (5,2);
\draw[thick] (5,0) -- (6,0);
\foreach \x/ \y in {0/0}
    {
    \draw[thick] (\x+1,\y) -- (\x,\y+1);
    \draw[color=white, line width=10] (\x,\y) -- (\x+1,\y+1);
    \draw[thick] (\x,\y) -- (\x+1,\y+1);
    }
\foreach \x/ \y in {5/1}
    {
    \draw[thick] (\x,\y) -- (\x+1,\y+1);
    \draw[color=white, line width=10] (\x+1,\y) -- (\x,\y+1);
    \draw[thick] (\x+1,\y) -- (\x,\y+1);
    }
\foreach \x/ \y in {1/1, 2/0, 4/0}
    {
    \draw[thick] (\x,\y) -- (\x+1,\y+1);
    \draw[color=white, line width=10] (\x+1,\y) -- (\x,\y+1);
    \draw[thick] (\x+1,\y) -- (\x,\y+1);
    }
\foreach \x/ \y in {2/0, 4/0}
    {
    \draw[thick] (\x+1,\y) -- (\x,\y+1);
    \draw[color=white, line width=10] (\x,\y) -- (\x+1,\y+1);
    \draw[thick] (\x,\y) -- (\x+1,\y+1);
    }
\foreach \x/ \y in {3/0}
    {
    \draw[thick] (\x+1,\y) -- (\x,\y+1);
    \draw[color=white, line width=10] (\x,\y) -- (\x+1,\y+1);
    \draw[thick] (\x,\y) -- (\x+1,\y+1);
    }
\end{tikzpicture}
\caption{\label{fig:knot} A 2-bridge knot with continued fraction [1,1,3,1] drawn horizontally instead of vertically for space so that the ``maxima'' are on the left}
\end{figure}
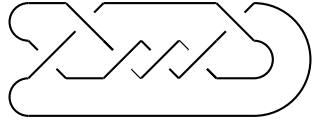

Instead of considering the set $K(c)$ of 2-bridge knots with crossing number $c$, Cohen \cite{Coh:lower} introduced a proxy set $T(c)$ of words corresponding to 2-bridge knot diagrams with crossing number $c$.  Since words and their reverses give two diagrams for the same knot, we abuse notation to refer to $T(c)$ as a partially-double-counted set of knots, where the the palindromic knots are those counted only once.  Cohen and Lowrance \cite{CohLow} used this combinatorial description to produce a convenient recursion  between $T(c)$ and $T(c-1)$ together with two copies of $T(c-2)$, and they computed the average Seifert 3-genus $\overline{g}(c)$ over $K(c)$.  See also independent work on this topic by Suzuki and Tran \cite{SuzukiTran}  and Ray and Diao \cite{RayDiao}.  
Cohen, Dinardo, Lowrance, Raanes, Rivera, Steindl, and Wanebo \cite{CohLowURSI} computed the number $t(c,g)$ of 2-bridge knots in $T(c)$ with crossing number $c$ and genus $g$.  They then used this calculation to show a central limit theorem over $K(c)$:  that the distribution of genera for 2-bridge knots with crossing number $c$ asymptotically approaches a normal distribution as $c$ approaches infinity.  Perhaps this suggests some sort of ``combinatorial nature'' for the set of 2-bridge knots of crossing number $c$.

We now tell the analogous story for the signature of a knot.  This is the signature, or difference between the numbers of positive and negative eigenvalues, of a matrix associated with the knot.  This matrix is related to the Seifert surface of the knot which is used to obtain the genus of the knot.  This invariant, introduced by Trotter \cite{Trotter} and given a diagrammatic formulation for alternating knots by Traczyk \cite{Tra} and Dasbach and Lowrance \cite{DasLow:sig}, plays an important role in comparing more difficult knot invariants.  We note that signature is always even.  For example, the knot in Figure \ref{fig:knot} has signature 0.  More background can be found in \cite{CLMR}.  Cohen, Lowrance, Madras, and Raanes \cite{CLMR} computed the average (absolute value of) signature $\overline{|\sigma|}(c)$ over all 2-bridge knots with crossing number $c$.  This group also developed recursions for the number $s(c,\sigma)$ of those knots in $T(c)$ with crossing number $c$ and signature $\sigma=2n$.  Small values are given in the following Table \ref{tab:sig}, which appears as Table 4 in \cite{CLMR}.

Below we will find closed formulas for these values and then use this result to show a central limit theorem for signatures of 2-bridge knots.  

\begin{theorem}
\label{thm:CLT}
For $c\geq 3$, let $K(c)$ be the set of 2-bridge knots of a given crossing number $c$ with each knot counted exactly once.  Let $k(c,\leq x)$ be the number of knots in $K(c)$ with signature at most $x$.  
For every real number $t$,
\begin{equation}
    \label{eq.CLTk}
    \lim_{c\rightarrow\infty}  \frac{k\left(c, \,\leq t\sqrt{c}\right)}{|K(c)|}   \;=\; \Phi(t),
\end{equation}
where $\Phi$ is the cumulative distribution  function of the standard normal distribution. 
That is, over the set $K(c)$ of 2-bridge knots with crossing number $c$, the distribution of signatures 
is asymptotically normally distributed  with mean 0 and variance $c$ as $c$ tends to infinity.
\end{theorem}

Using an important result on quasi-isomorphisms by Bj\"orklund and Hartnick \cite{BjoHar}, Dunfield and Tiozzo \cite{DunTio} prove a central limit theorem for the signature of random positive 3-braids.  Again, this seems to suggest some sort of ``combinatorial nature" for the set of 2-bridge knots with $c$ crossings. 
Might our Theorem \ref{thm:CLT} together with this result shed light on the behavior of  signatures of knots in general?  More generally, perhaps the approach by Dunfield and Tiozzo may lead to a more unified set of central limit theorems for knot invariants.  
 The key may be to have an explicit mechanism that builds up classes of knots by a sequence of
simple steps, possibly with constraints.

The proof of our Main Theorem \ref{thm:CLT} will follow from these closed formulas for $s(c,\sigma)$.

\begin{theorem}
\label{thm:sums}
For $c\geq 3$, let $T(c)$ be the partially-double-counted set of 2-bridge knots with crossing number $c$ where the palindromic knots are those counted only once.

When $c=2m$ is even, the number of 2-bridge knots with crossing number $c$ and signature $2n$ from among the set $T(c)$ is 
\begin{equation}
\label{eq:even}
  s(2m,2n)=\displaystyle\sum_{i=0}^{m-|n|-2}{2m-4-2i\choose m+n-2-i}.  
\end{equation}    
When $c=2m+1$ is odd, the number of 2-bridge knots with crossing number $c$ and signature $2n$ from among the set $T(c)$ is 
\begin{equation}
\label{eq:odd}
    s(2m+1,2n)=\sum_{i=0}^{\min\{m+n-3,\; m-n\}}{2m-3-2i\choose m+n-3-i}
\end{equation}
for all $n \neq 1$, 
and when $n=1$ we have
\begin{equation}
\label{eq:odd2}
    s(2m+1,2)=1+\sum_{i=0}^{m-2}{2m-3-2i\choose m-2-i},
\end{equation}
which is just one more than the formula given in Equation \eqref{eq:odd}.

These are sums of binomial coefficients in columns of (the centered) Pascal's triangle.
\end{theorem}

For more on this last sentence, see Example \ref{ex:10} below.

For similar asymptotic work on non-orientable 3-genus and crosscap number, see Cohen, Kindred, Lowrance, Shanahan, and Van Cott \cite{CKLSVC}.  For similar asymptotic work on the Casson invariant of 2-bridge knots, see Cohen and Ville \cite{CohVil}.  For similar asymptotic work by other authors on the braid index of 2-bridge knots, see Clark, Frank, and Lowrance \cite{LowURSI:braids} and independently Suzuki and Tran \cite{SuzTra:braids}.

\subsection*{Organization}  The reader can find relevant information from the previous paper \cite{CLMR} on signature in Section 2, including Table \ref{tab:sig} of values of $s(c,\sigma)$ and some results that will be used later.

In Section 3 we prove Theorem \ref{thm:sums} using an unusual induction proof involving parity.

In Section 4 we prove the central limit theorem for $K(c)$ of Theorem \ref{thm:CLT} by first proving in Theorem \ref{thm-CLT} a central limit theorem for $T(c)$.
 
In Section 5, we provide some numerical computations of variance.

\subsection*{Acknowledgements}  The research of NM was supported in part by a Discovery Grant from NSERC of Canada.  The authors would like to thank the Research, Scholarship, and Creative Activities (RSCA) office at the State University of New York at New Paltz for the Academic Year Undergraduate Research Experience (AYURE) recognition that CB and DT received under their professor MC.  Undergraduate students CB, HD, RF, and DT enrolled in MAT 490 Mathematics Research with MC, and RS was involved in this group through his high school's Science Research program.

\section{Background}

In this section we briefly provide relevant results from previous works that will be useful below.

Ernst and Sumners tell us the number of 2-bridge knots, with chiral pairs not counted separately:

\begin{theorem}[Ernst-Sumners \cite{ErnSum}, Theorem 5]
\label{thm:ernstsumners}
For $c\geq 3$, the number $|K(c)|$ of 2-bridge knots with crossing number $c$ where chiral pairs are \emph{not} counted separately is given by
\[
|K(c)| = 
\begin{cases}
\frac{1}{3}(2^{c-3}+2^{\frac{c-4}{2}}) & \text{ for } c\equiv 0 \text{ mod }4,\\
\frac{1}{3}(2^{c-3}+2^{\frac{c-3}{2}}) & \text{ for } c\equiv 1 \text{ mod }4, \\
\frac{1}{3}(2^{c-3}+2^{\frac{c-4}{2}}-1) & \text{ for } c\equiv 2 \text{ mod }4, \text{ and}\\
\frac{1}{3}(2^{c-3}+2^{\frac{c-3}{2}}+1) & \text{ for } c\equiv 3 \text{ mod }4.
\end{cases}
\]
\end{theorem}

Proposition 3.1 of \cite{CohLow} counts the size $|T(c)|$ of the proxy partially-double-counted set of 2-bridge knots with crossing number $c$:

\begin{proposition}
\cite[Proposition 3.1]{CohLow}
\label{prop:CohLowT}
The number $|T(c)| = \frac{1}{3}(2^{c-2} - (-1)^c)$ of knots in the partially-double-counted set $T(c)$ is the Jacobsthal number $J(c-2)$ satisfying the recursive formula $t(c)=t(c-1)+2t(c-2)$ with initial values $t(2)=J(0)=0$ and $t(3)=J(1)=1$.
\end{proposition}

  These terms terms appear as A001045 in the On-Line Encyclopedia of Integer Sequences (OEIS) \cite{OEIS}.

The end of the proof of the Main Theorem \ref{thm:CLT} utilizes the following set $T_p(c)$ which is defined more easily in \cite{CohLow} in terms of words in $\{+,-\}$.

\begin{definition}
\label{def:palindrome}
A palindromic 2-bridge knot in $T(2m)$ has a diagram that can be rotated 180 degrees in the plane so that the rotated diagram is the mirror image of the original, which by definition makes these knots  amphichiral.    All of these knots have signature 0.

A palindromic 2-bridge knot in $T(2m+1)$ can be rotated 180 degrees out of the plane about a vertical axis in the plane such that the diagram is unaltered.

The palindromic knots $T_p(c)$ with crossing number $c$ are exactly those counted exactly once in $T(c)$ so that by construction two copies of the set $K(c)$ can be partitioned into $T(c)$ and $T_p(c)$.
\end{definition}

Just prior to our proof of Theorem \ref{thm:CLT} at the end of Section \ref{sec:CLT}, we will use Theorem \ref{thm:ernstsumners} and Proposition \ref{prop:CohLowT} to show that the size of $T_p(c)$ is negligible in comparison to $T(c)$.

Recall that $s(c,\sigma)$ counts the number of 2-bridge knots in the partially double-counted set $T(c)$ with crossing number $c$ and signature $\sigma$.  We turn these values into a triangle that we refer to as the \textbf{signature triangle}.  The following Table \ref{tab:sig} comes from Table 4 of \cite{CLMR}, obtained from a recursive result that we omit here.

\begin{table}[h!]
\begin{tabular}{c|cccc|c|c|c|ccccc}
$c\backslash\sigma$ & -10 & -8 & -6 & -4 & -2 & 0 & 2 & 4 & 6 & 8 & 10 & 12 \\
\hline
3 &   &   &   &  &  &  & 1 &  &  &  &  &  \\
4 &   &   &   &  &  & 1 &  &  &  &  &  &  \\
s5 &   &   &   &  &  &  & 2 & 1 &  &  &  &  \\
6 &  &  &  &   & 1 & 3 & 1 &  &  &  &  &  \\
7 &  &  &  &  &  & 1 & 5 & 4 & 1 &  &  &  \\
8 &  &  &  & 1 & 5 & 9 & 5 & 1 &  &  &  &  \\
9 &  &  &  &  & 1 & 6 & 15 & 14 & 6 & 1 &  &  \\
10 &  &  & 1 & 7 & 20 & 29 & 20 & 7 & 1 &  &  &  \\
11 &  &  &  & 1 & 8 & 27 & 50 & 49 & 27 & 8 & 1 &  \\
12 &  & 1 & 9 & 35 & 76 & 99 & 76 & 35 & 9 & 1 &  &  \\
13 &  &  & 1 & 10 & 44 & 111 & 176 & 175 & 111 & 44 & 10 & 1 \\
14 & 1 & 11 & 54 & 155 & 286 & 351 & 286 & 155 & 54 & 11 & 1 & 
\end{tabular}
\caption{Numbers $s(c,\sigma)$ of knots in $T(c)$ with crossing number $c$ and signature $\sigma$, as in Table 4 of \cite{CLMR}}
\label{tab:sig}
\end{table}

We note here that some of these entries can be found in the On-Line Encyclopedia of Integer Sequences (OEIS) \cite{OEIS}.  For even rows, the central terms for $\sigma=0$ appear as A006134, the terms for $\sigma=2$ appear as A057552, and the terms for $\sigma=4$ appear as A371964.  
For odd rows the terms for $\sigma=4$ appear as A079309, and the terms for $\sigma=6$ appear as  A371965.

It is not a coincidence that this triangle has some similarities to Pascal's triangle, which has also been referred to as  Khayyam's triangle in Iran,  Yang Hui's triangle in China, Tartaglia's triangle in Italy, and the Staircase of Mount Meru in India.

Lemma 3.6 from \cite{CLMR} shows that the coefficients in our signature triangle for $T(c)$ behave very similarly to entries in Pascal's triangle, with some slight deviation when $\sigma=\pm 2$:

\begin{lemma}
\cite[Lemma 3.6]{CLMR}
\label{lem:signaturerecursion2}
If $c\geq 4$, the number $s(c, \sigma)$ of words in $T(c)$ corresponding to a knot with signature $\sigma$ satisfies a different recurrence relation similar to that for binomial coefficients:
\begin{equation}
\label{eq:recursion2}
s(c,\sigma)=
\begin{cases}
s(c-1,\sigma-2)+s(c-1,\sigma-4) & \text{when $c$ is odd and $\sigma\neq 2$,}\\
s(c-1,\sigma-2)+s(c-1,\sigma-4)+1 & \text{when $c$ is odd and $\sigma= 2$,}\\
s(c-1,\sigma+2)+s(c-1,\sigma+4) & \text{when $c$ is even and $\sigma\neq -2$, and}\\
s(c-1,\sigma+2)+s(c-1,\sigma+4)-1 & \text{when $c$ is even and $\sigma= -2$.}\\
\end{cases}
\end{equation}
\end{lemma}

Proposition 3.7 in \cite{CLMR} describes the symmetry among even and odd rows of the signature triangle, with a notable exception in the $\sigma=2$ column:

\begin{proposition}
\cite[Proposition 3.7]{CLMR}
\label{prop:sym}
    The number of $s(c,\sigma)$ of words in $T(c)$ satisfies the following symmetries:
    \begin{alignat*}{3}
        s(c,\sigma) = & \; s(c,-\sigma) & \quad &\text{if $c$ is even,}\\
    s(c,2) = & \; s(c,4)+1 & \quad & \text{if $c$ is odd and $\sigma$ is 2 or 4, and}\\
        s(c,\sigma)  = & \;   s(c,6-\sigma)& \quad & \text{if $c$ is odd and $\sigma\geq 6$~\text{or}~$\sigma\leq0$.}    
    \end{alignat*}
\end{proposition}

Theorem 3.8 from \cite{CLMR} even demonstrates that consecutive rows sum to binomial coefficients, although only one case was shown:

\begin{theorem}
\cite[Theorem 3.8]{CLMR}
\label{thm:binomials}
Let $T_m=T(2m+1)\cup T(2m+2)$.  
Then there are exactly 
$\binom{2m-1}{k}=\binom{2m-1}{m-1+\frac{\sigma}{2}}$  knots in $T_m$
with signature $\sigma=2k-2m+2$. In other words, 
\[s(2m+1,\sigma) + s(2m+2,\sigma) = \binom{2m-1}{k}.\]
\end{theorem}

The alternative case is also true with a similar proof.  For $k=m-2+\frac{\sigma}{2}$ or $\sigma=2k-2m+4$,
\[s(2m,\sigma) + s(2m+1,\sigma) = \binom{2m-2}{k}.\]

\section{Proof of Theorem \ref{thm:sums}}
\label{sec:sums}

Even though it is not required for the proof, we provide information for the calculation of the case where $c=10$ in order to illustrate the beauty of Theorem \ref{thm:sums}.

\begin{example}
    \label{ex:10}
When $c=10$, we refer to the 0th through 6th rows of Pascal's triangle.  To get every entry of Table \ref{tab:sig} for $c=10$, we sum the entries in Pascal's triangle vertically as in Table \ref{tab:pascal}.

For example,  in the center column, we have   $29=s(10,0)=\binom{6}{3}+\binom{4}{2}+\binom{2}{1}+\binom{0}{0}=20+6+2+1$.
\end{example}

\begin{table}[h!]
\begin{tabular}{ccccccc|c|cccccc}
 &   &   &   &  &  &  & 1 &  &  &  &  &  &  \\
 &   &   &   &  &  & 1 &   & 1 &  &  &  &  &  \\
 &   &   &   &  & 1 &  & 2 &  & 1 &  &  &  &  \\
 &  &  &  &  1 &  & 3 &  & 3 &  & 1  &  &  &  \\
 &  &  & 1 &  & 4 &  & 6 &  & 4 &  & 1 &  &  \\
 &  & 1 &  & 5 &  & 10 &  & 10 &  & 5 &  & 1 &  \\
 & 1 &  & 6 &  & 15 &  & 20 &  & 15 &  & 6 &  &  1 \\
 \hline
  & 1 &  & 7 &  & 20 &  & 29 &  & 20 &  & 7 &  &  1 
\end{tabular}
\caption{\label{tab:pascal} The final row is $s(10,\sigma)$ with the first seven rows of Pascal's triangle above it.  Each entry in this row is a sum of the binomial coefficients above it.}
\end{table}

We establish the base cases for the induction argument below in the following example.

\begin{example}
    \label{ex:basecases}
When $m=1$ and $c=2m+1=3$,  
we have from Table \ref{tab:sig} that $s(3,2)=1$ with all other $s(c,\sigma)=0$.  Observe that the right hand side of Equation \eqref{eq:odd}
\[    \sum_{i=0}^{m+n-3}{2m-3-2i\choose m+n-3-i}= \sum_{i=0}^{n-2}{-1-2i\choose n-2-i}\]
is 0 for all $n$.  However, Equation \eqref{eq:odd2} for $n=1$ is one more than this by the formula.

When $m=2$ and $c=2m=4$, we have $s(4,0)=1$ with all other $s(c,\sigma)=0$.  Observe that the right hand side of Equation \eqref{eq:even}
\[\sum_{i=0}^{m-|n|-2}{2m-4-2i\choose m+n-2-i}=\sum_{i=0}^{-|n|}{-2i\choose n-i}\]
is non-zero exactly when $n=0$, in which case we have just the single term  $\sum_{i=0}^{0}{-2i\choose -i}={0\choose 0}=1$.

When $m=2$ and $c=2m+1=5$, we have $s(5,2)=2$ and $s(5,4)=1$ with all other $s(5,\sigma)=0$.  Observe that the right hand side of Equation \eqref{eq:odd}
\[    \sum_{i=0}^{m+n-3}{2m-3-2i\choose m+n-3-i}= \sum_{i=0}^{n-1}{1-2i\choose n-1-i}\]
is non-zero both for $n=1$, in which case we get the single term $\sum_{i=0}^{0}{1-2i\choose -i}={1 \choose 0}$=1, and for $n=2$, in which case we get the single term (curtailing the sum at $0$) $\sum_{i=0}^{0}{1-2i\choose 1-i}={1 \choose 1} =1$.  Note that Equation \eqref{eq:odd2} gives that $s(5,2)$ is one more than the first calculation.
\end{example}

We now perform an unusual induction proof on the crossing number $c$ in which the odd inductive hypothesis proves the even induction and in which the even inductive hypothesis proves the odd induction.  This technique is described on page 69 of the textbook Discrete Mathematics:  Elementary and Beyond by Lov{\'a}sz, Pelik{\'a}n, and Vesztergombi \cite{lovasz2003discrete} for the pair of Fibonacci relations $F_{2n-1}=F_n^2+F_{n-1}^2$ and $F_{2n}=F_{n+1}F_{n}+F_{n}F_{n-1}$.

\subsection{Assuming odd crossing number to prove even crossing number}
\begin{proof}
Assume Equation \eqref{eq:odd} for odd crossing number
\[ s(2m+1,2n)=\sum_{i=0}^{\min\{m+n-3,\; m-n \}}{2m-3-2i\choose m+n-3-i}\] for all $n \neq 1$, where the sum need only run until $m-n$ for $n\geq 1$, and when $n=1$ we have Equation \eqref{eq:odd2}
\[s(2m+1,2)=1+\sum_{i=0}^{m-2}{2m-3-2i\choose m-2-i}.\]
We will show:
\[  s(2m+2,2n)=\displaystyle\sum_{i=0}^{m-|n|-1}{2m-2-2i\choose m+n-1-i}.  \]

When $n\geq 1$, using Lemma \ref{lem:signaturerecursion2} 
and that $\sigma=2n\ne-2$, followed by the inductive hypothesis, we have that 
\begin{align*}
s(2m+2,2n) & =s(2m+1,2n+2)+s(2m+1,2n+4) \\
 & =\sum_{i=0}^{m-n-1}{2m-3-2i\choose m+n-2-i} +\sum_{i=0}^{m-n-2}{2m-3-2i\choose m+n-1-i} \\
 &= \left[ {2m-3-2(m-n-1)\choose m+n-2-(m-n-1)} + \sum_{i=0}^{m-n-2}{2m-3-2i\choose m+n-2-i} \right] +\sum_{i=0}^{m-n-2}{2m-3-2i\choose m+n-1-i} \\
 &=  {2n-1\choose 2n-1} + \sum_{i=0}^{m-n-2} \left[{2m-3-2i\choose m+n-2-i} + {2m-3-2i\choose m+n-1-i} \right] \\
 &=  {2m-2-2(m-n-1)\choose m+n-1-(m-n-1)} + \sum_{i=0}^{m-n-2} \left[{2m-2-2i\choose m+n-1-i}  \right] \\
 & =  \sum_{i=0}^{m-n-1} \left[{2m-2-2i\choose m+n-1-i}  \right],
\end{align*}
as desired.

When $n=0$, following the same lemma but following the alternative inductive hypothesis,
\begin{align*}
   s(2m+2,0) & = s(2m+1,2)+s(2m+1,4) \\
   & = \left[ 1+\sum_{i=0}^{m-2}{2m-3-2i\choose m-2-i}\right] +\sum_{i=0}^{m-2}{2m-3-2i\choose m-1-i} \\
  & ={0 \choose 0}+\sum_{i=0}^{m-2} \left[{2m-3-2i\choose m-2-i} +{2m-3-2i\choose m-1-i} \right] \\ 
  & = {2m-2-2(m-1) \choose m-1-(m-1) }+\sum_{i=0}^{m-2} \left[{2m-2-2i\choose m-1-i} \right] \\ 
  & = \sum_{i=0}^{m-1} {2m-2-2i\choose m-1-i},
\end{align*}
as desired.

Lastly, when $n<0$, we use Proposition \ref{prop:sym} to get
\[
s(2m+2,2n) = s(2m+2,-2n) = \sum_{i=0}^{m-|-n|-1} {2m-2-2i \choose m-n-1-i} = \sum_{i=0}^{m-|n|-1} {2m-2-2i \choose m-n-1-i}.
\]

Since ${p \choose k} = {p \choose p-k}$ with $p=2(m-1-i)$ and $k=(m-1-i)-n$, we have that 
\[
s(2m+2,2n)= \sum_{i=0}^{m-|n|-1} {2m-2-2i \choose m+n-1-i},
\]
completing this part of the inductive argument.
\end{proof}

\subsection{Assuming even crossing number to prove odd crossing number}
\begin{proof}
Assume Equation \eqref{eq:even} for even crossing number
\[ s(2m,2n)=\sum_{i=0}^{m-|n|-2}{2m-4-2i \choose m+n-2-i}.\]
We will show:
\[s(2m+1,2n)=\sum_{i=0}^{\min\{m+n-3,m-n\}}{2m-3-2i\choose m+n-3-i}\]
for all $n \neq 1$, where the sum need only run until $m-n$ for $n\geq 2$, and when $n=1$ we'll show
\[s(2m+1,2)=\sum_{i=0}^{m-2}{2m-3-2i\choose m-2-i}+1.\]

When $n\geq 2$, using Lemma \ref{lem:signaturerecursion2}  we have that 
\begin{align*}
    s(2m+1,2n) & =s(2m,2n-2)+s(2m,2n-4) \\
    & = \sum_{i=0}^{m-|n-1|-2}{2m-4-2i\choose m+n-3-i} +\sum_{i=0}^{m-|n-2|-2}{2m-4-2i\choose m+n-4-i} \\
    & = \sum_{i=0}^{m-(n-1)-2}{2m-4-2i\choose m+n-3-i} +\sum_{i=0}^{m-(n-2)-2}{2m-4-2i\choose m+n-4-i} \\
    & = \sum_{i=0}^{m-n-1}{2m-4-2i\choose m+n-3-i} +\sum_{i=0}^{m-n}{2m-4-2i\choose m+n-4-i} \\
    & =  \sum_{i=0}^{m-n-1}{2m-4-2i\choose m+n-3-i}  +\left[ \sum_{i=0}^{m-n-1}{2m-4-2i\choose m+n-4-i} +  {2m-4-2(m-n)\choose m+n-4-(m-n)}\right]\\
    & = \sum_{i=0}^{m-n-1}\left[ {2m-4-2i\choose m+n-3-i} + {2m-4-2i\choose m+n-4-i}\right] + {2n-4\choose 2n-4} \\
    & = \sum_{i=0}^{m-n-1}\left[ {2m-3-2i\choose m+n-3-i}\right] + {2m-3-2(m-n)\choose m+n-3-(m-n)} \\
    & = \sum_{i=0}^{m-n}{2m-3-2i\choose m+n-3-i},
\end{align*}
as desired.

When $n=1$, Proposition \ref{prop:sym} tells us that $s(2m+1,2)=s(2m+1,4)+1$. 
We use again that ${p \choose k} = {p \choose p-k}$ with $p=2m-3-2i$ and $k=m-1-i$ to yield the desired result.

When $n\leq 0$, using Lemma \ref{lem:signaturerecursion2} 
we have that 
\begin{align*}
    s(2m+1,2n) & =s(2m,2n-2)+s(2m,2n-4) \\
    & = \sum_{i=0}^{m-|n-1|-2}{2m-4-2i\choose m+n-3-i} +\sum_{i=0}^{m-|n-2|-2}{2m-4-2i\choose m+n-4-i} \\
    & = \sum_{i=0}^{m+(n-1)-2}{2m-4-2i\choose m+n-3-i} +\sum_{i=0}^{m+(n-2)-2}{2m-4-2i\choose m+n-4-i} \\
    & = \sum_{i=0}^{m+n-3}{2m-4-2i\choose m+n-3-i} +\sum_{i=0}^{m+n-4}{2m-4-2i\choose m+n-4-i} \\
    & = \left[ {2m-4-2(m+n-3)\choose m+n-3-(m+n-3)} +\sum_{i=0}^{m+n-4}{2m-4-2i\choose m+n-3-i}\right] +\sum_{i=0}^{m+n-4}{2m-4-2i\choose m+n-4-i} \\
    & = {2-2n\choose 0} +\sum_{i=0}^{m+n-4}\left[ {2m-4-2i\choose m+n-3-i} + {2m-4-2i\choose m+n-4-i}\right] \\
    & = {2m-3-2(m+n-3)\choose m+n-3-(m+n-3)} +\sum_{i=0}^{m+n-4}\left[ {2m-3-2i\choose m+n-3-i}\right] \\
    & = \sum_{i=0}^{m+n-3}{2m-3-2i\choose m+n-3-i},
\end{align*}
as desired.
\end{proof}

Together with the base cases discussed in Example \ref{ex:basecases}, these two inductive steps complete the proof of Theorem \ref{thm:sums}.

\section{Proof of Theorem \ref{thm:CLT}}
\label{sec:CLT}

Before we get to the proof of the central limit theorem for $K(c)$ we will prove a central limit theorem for $T(c)$.  First we need to normalize $s(c,\sigma)$ over all $\sigma$.  
Recall from Proposition \ref{prop:CohLowT} that
the total number of knots in $T(c)$ is    $\frac{1}{3} \left(2^{c-2}-(-1)^c\right)$, 
and hence 
\begin{equation}
   \label{eq.Tlim}
   \lim_{c\rightarrow\infty} \frac{2^{c}}{|T(c)|}   \;=\;  12\,.  
\end{equation}

\begin{definition}
For real $t$, we define the cumulative sum $s(c,\,\leq t)$ to be $\sum_{j:\,j\leq t}  s(c,j)$.  
We write $\Phi$ for  the cumulative distribution function of the standard normal distribution
\[   \Phi(t)  \;=\;  \frac{1}{\sqrt{2\pi}} \int_{-\infty}^t e^{-x^2/2}\,dx\,.
\]
\end{definition}

\begin{theorem}[Central Limit Theorem for $T(c)$]
   \label{thm-CLT}
For every real number $t$,
\begin{equation}
    \label{eq.CLT2m}
    \lim_{c\rightarrow\infty}  \frac{s\left(c, \,\leq t\sqrt{c}\right)}{|T(c)|}   \;=\; \Phi(t).
\end{equation}
That is, for knots in $T(c)$, the distribution of signatures is asymptotically normally distributed  with mean 0 and variance $c$ as $c$ tends to infinity.
\end{theorem}

The following property will be useful in the proof.  

\begin{lemma}
   \label{lem-unifcont}
Let $G_1,G_2,\ldots$ be a sequence of cumulative distribution functions such that\\ $\lim_{n\rightarrow\infty}G_n(t) \,=\,\Phi(t)$ for every $t\in \mathbb{R}$.  
Let $\{u_n\}$ be a sequence of real numbers that converges to $u\in \mathbb{R}$.  Then $\lim_{n\rightarrow\infty}G_n(u_n) \,=\, \Phi(u)$.
\end{lemma}

\begin{proof}
Assume $\{u_n\}$ converges to $u$.  Let $\epsilon>0$.  Since $\Phi$ is continuous, we can choose 
real $a$ and $b$ such that $a<u<b$ and
\[   \Phi(u)-\frac{\epsilon}{2} < \Phi(a)  <   
     \Phi(b) < \Phi(u)+\frac{\epsilon}{2}\,.
\]
By our hypotheses, we can choose $N_0$ such that $\Phi(u)-\epsilon <G_n(a)$, $G_n(b)<\Phi(u)+\epsilon$, and $u_n\in [a,b]$ for all $n\geq N_0$.  
Then for all $n\geq N_0$ we have $G_n(a)\leq G_n(u_n)\leq G_n(b)$ and hence $\Phi(u)-\epsilon < G_n(u_n) <\Phi(u)+\epsilon$.
Since $\epsilon$ is an arbitrary positive number, the lemma follows.
\end{proof}

\begin{proof}[Proof of Theorem \ref{thm-CLT}]
We begin with case that $c$ is even, and write $c=2m$.  
Since $s(2m,-2n)=s(2m,2n)$, it suffices to consider $n\leq 0$ in the following and prove Equation \eqref{eq.CLT2m}
for $t\leq 0$.

\begin{definition}
For real $x$ and natural number $k$, we define the cumulative sum over the binomial coefficients $\mbox{Bin}(k,\leq x)$ to be $\sum_{j:j\leq x}  \binom{k }{j}$.
\end{definition}

Recall that for $n\leq 0$, Theorem \ref{thm:sums} says that the number $s(2m,2n)$ of 2-bridge knots with crossing number $c$ and signature $2n$ from the set $T(c)$ is $\sum_{i=0}^{m+n-2}   \binom{ 2m-2i-4}{m+n-2-i}.$    
Example \ref{ex:10} and Table \ref{tab:pascal} interpret this term as the truncated sum of entries in a column of Pascal's triangle.  

From this, since $m+n-2-i\leq (2m-2i-4)/2$, we obtain
\begin{equation}
   \label{eq.ssumcum}
   s(2m,\leq 2n)   \;= \;   \sum_{i=0}^{m+n-2}  \mbox{Bin}\left(2m-2i-4,\,\leq m+n-i-2\right).
\end{equation}

Now we bring in probability.  

\begin{definition}
    For each positive integer $M$, let $X_M$ be a binomially distributed random variable with 
parameters $M$ and $1/2$, i.e.  for $k=0,1,\ldots,M$, we have
\begin{equation}
    \label{eq.XNdef} 
       \Pr(X_M=k)  \;=\;  \binom{M}{k}  \frac{1}{2^{M}}.
\end{equation}
Then $X_M$ has mean $M/2$ and variance $M/4$.  Accordingly, we let $X_M^*$ be the ``standardized'' version of $X_M$, 
i.e.
\[     X_M^*  \; :=\;   \frac{X_M-M/2}{\sqrt{M/4}}  \,,
\] 
and let $F_M$ be the cumulative distribution function of $X_M^*$:
\begin{equation}
   \label{eq.FNdef}
       F_M(t)  \;:=\;  \Pr(X_M^*\leq t)   \;=\;  \Pr\left(X_M\leq  (M/2)+t\sqrt{M/4}\right).
\end{equation}

Now consider the terms in the sum of Equation \eqref{eq.ssumcum}.  
Define
\begin{equation}
   \label{eq.rhodef}
       \rho_{m,i}(n)   \; :=\;   \frac{ \mbox{Bin}(2m-2i-4,\,\leq m+n-i-2)}{2^{2m-2i-4}}.
\end{equation}
\end{definition}

We then observe that 
\begin{eqnarray*}
    \rho_{m,i}(n)   & = &  \sum_{k:\,k\,\leq\, m+n-i-2}  \binom{2m-2i-4}{k} \, \frac{1}{2^{2m-2i-4}}
    \\
    & = &   \Pr\left(X_{2m-2i-4}\,\leq  (m-i-2)+n\right)   \hspace{5mm}\mbox{(by Equation \eqref{eq.XNdef})}
    \\
      & = &   \Pr\left(   \frac{X_{2m-2i-4}-(m-i-2)}{\sqrt{(m-i-2)/2}} \; \leq \;  \frac{n}{\sqrt{(m-i-2)/2}}\right)
      \\
      & = & F_{2m-2i-4}\left(  \frac{n\sqrt{2}}{\sqrt{m-i-2}}\right)    \hspace{5mm}\mbox{(by Equation \eqref{eq.FNdef})}.
\end{eqnarray*}

The classical De Moivre - Laplace Central Limit Theorem tells us that for every $t\in\mathbb{R}$ we have
\begin{equation}
   \label{eq.binCLT}
      \lim_{M\rightarrow\infty}F_M(t) \;=\; \Phi(t).
\end{equation}

Now consider $n$ as a function of $m$, specifically $n=\lfloor t\sqrt{m/2}\rfloor$ for a constant $t\leq 0$.  This choice gives the right scaling for the Central Limit Theorem.  In particular, 
since $n$ grows with $m$, we have for every fixed $i$ that     
\[    \lim_{ m\rightarrow\infty} \frac{n\sqrt{2}}{\sqrt{m-i-2}}   \;=\;  t\,.
\]
Thus it follows from Equation \eqref{eq.binCLT} and Lemma \ref{lem-unifcont}  that for every $i$ and every $t\leq 0$
\begin{equation}
    \label{eq.rhoconv}
      \lim_{m\rightarrow\infty}\rho_{m,i}\left(\lfloor t\sqrt{m/2}\rfloor \right)   \;=\;  \Phi(t).
\end{equation}
Now for $n=\lfloor t\sqrt{m/2}\rfloor$ with $t\leq 0$, we have
\begin{eqnarray}
   \nonumber
     \frac{s(2m,\leq 2n)}{|T(2m)|}  & = & \frac{2^{2m-4}}{|T(2m)|}\sum_{i=0}^{m+n-2}  
        \frac{   \mbox{Bin}(2m-2i-4,\leq m+n-i-2) }{ 2^{2i} \times 2^{2m-2i-4}}
     \\
     \nonumber
       & &    \hspace{15mm}\hbox{(by Equation \eqref{eq.ssumcum})}
     \\
     \nonumber
     & = & \frac{2^{2m-4}}{|T(2m)|}\sum_{i=0}^{m+n-2}  \frac{1}{2^{2i}} \, \rho_{m,i}(n)
                \hspace{8mm}\hbox{(by Equation \eqref{eq.rhodef})}.
\end{eqnarray}
Using the above expression together with Equations \eqref{eq.Tlim} and \eqref{eq.rhoconv} and the Dominated Convergence 
Theorem, we obtain for  $n=\left\lfloor t\sqrt{m/2}\right\rfloor\leq 0$ that
\begin{equation*}
   \label{eq.soverTlim}
     \lim_{m\rightarrow\infty}  \frac{s\left(2m,\leq t\sqrt{2m}\right)}{|T(2m)|} 
      \;=\;   \lim_{m\rightarrow\infty}  \frac{s(2m,\leq 2n)}{|T(2m)|}  \;=\;     
      \frac{3}{4}\sum_{i=0}^{\infty}\frac{1}{4^i}\,\Phi(t)  
      \;=\; \frac{3}{4} \,\frac{1}{1-\frac{1}{4}} \,\Phi(t) \;=\;  \Phi(t) \,,
\end{equation*}
where for the first equality, note that $2n\leq t\sqrt{2m} < 2n+2$.  
This proves Equation \eqref{eq.CLT2m} for the subsequence of even $c$.

Now we consider the case that $c$ is odd, and we write $c=2m+1$.   Fix $t\in \mathbb{R}$.  
By Theorem \ref{thm:binomials}, we have
for every $k$ that 
\begin{equation}
   \label{eq.oddtoeven}
   s(2m+1,2k-2m+2)   \;=\;   \binom{2m-1}{k}  \;-\;  s(2m+2,2k-2m+2) \,.
\end{equation}
Summing Equation \eqref{eq.oddtoeven} over all $k$ less than or equal to $(t/2)\sqrt{2m+1}+m-1$ gives
\begin{equation}
   \label{eq.oddtoevensum}
   \frac{s(2m+1,\,\leq t\sqrt{2m+1})}{|T(2m+1)|}  \;=\;   A_m(t)   \;-\, B_m(t), 
\end{equation}
where
\begin{equation*}
     \label{eq.Amtdef}
     A_m(t)  \; = \;   \frac{2^{2m-1}}{|T(2m+1)|}  \frac{ \mbox{Bin}(2m-1, \,\leq (t/2)\sqrt{2m+1}-m+1)}{2^{2m-1}}
\end{equation*}
and
\begin{equation}
     \label{eq.Bmtdef}
     B_m(t)  \; = \;   \frac{|T(2m+2)|}{|T(2m+1)|}  \frac{s(2m+2, \,\leq t\sqrt{2m+1})}{|T(2m+2)|}\,.
\end{equation}
Similarly to the case of even $c$, we find that 
\begin{eqnarray*}    A_m(t)  & = & 
   \frac{2^{2m-1}}{|T(2m+1)|} \,
  \Pr\left(X_{2m-1}-\frac{1}{2}(2m-1) \,\leq\, \frac{t}{2} \sqrt{2m+1}+\frac{1}{2}\right) 
   \\
   &=& \frac{2^{2m-1}}{|T(2m+1)|} \,F_{2m-1}\left(\frac{(t/2)\sqrt{2m+1} +(1/2)}{\sqrt{(2m-1)/4}}\right),
\end{eqnarray*}
and hence (using Equation \eqref{eq.Tlim} and Lemma \ref{lem-unifcont}) that
\begin{equation}
    \label{eq.Amtlim}
    \lim_{m\rightarrow\infty}A_m(t)  \;=\; \frac{12}{4} \,\Phi(t) \,.
\end{equation}
From Equations \eqref{eq.Tlim} and \eqref{eq.Bmtdef} and the fact that Equation \eqref{eq.CLT2m} holds for the subsequence of even $c$, 
we can apply Lemma \ref{lem-unifcont} to obtain 
 \begin{equation}
    \label{eq.Bmtlim}
    \lim_{m\rightarrow\infty}B_m(t)  \;=\; 2 \,\Phi(t) \,.
\end{equation}   
It now follows from Equations \eqref{eq.oddtoevensum}, \eqref{eq.Amtlim} and \eqref{eq.Bmtlim} that
\begin{equation*}
    \label{eq.soddlim}
    \lim_{m\rightarrow\infty}   \frac{s(2m+1,\,\leq t\sqrt{2m+1})}{|T(2m+1)|} \;=\; 3\Phi(t)\,-\,2\Phi(t)  \;=\; \Phi(t).
\end{equation*}
Therefore Equation \eqref{eq.CLT2m} holds for the subsequence of odd $c$, and the proof of Theorem \ref{thm-CLT} is complete. 
\end{proof}

Now we turn to the Central Limit Theorem for  the set $K(c)$ of 2-bridge knots with crossing number $c$.    

\begin{definition}
    Let $k(c,\sigma)$ be the number of knots in $K(c)$ with signature $\sigma$, and let $k(c,\,\leq x)$ be the number of knots in $K(c)$ 
with signature at most $x$.

Let $T_p(c)$ be the set of palindromic knots with crossing number $c$, let $s_p(c,\sigma)$ be the number of knots in $T_p(c)$ with signature $\sigma$, 
and let $s_p(c,\,\leq x)$ be the number of knots in $T_p(c)$ with signature at most $x$.   
  By construction, for every $c$ and $\sigma$,
\begin{equation*}
   \label{eq.KTT}
   2\, |K(c)|  \;=\;  |T(c)| \,+\,  |T_p(c)|    \hspace{5mm}\hbox{and}\hspace{5mm}   k(c,\sigma)  \;=\; \frac{1}{2}\left(s(c,\sigma)+s_p(c,\sigma)\right).
\end{equation*}
\end{definition}

From Theorem \ref{thm:ernstsumners} and Proposition \ref{prop:CohLowT}, we find
\begin{equation}
   \label{eq.KTlim}
   \lim_{c\rightarrow\infty}  \frac{|T(c)|}{2|K(c)|}  \;=\; 1.
\end{equation}

We can now prove the Main Theorem \ref{thm:CLT}.

\begin{proof}[Proof of Theorem \ref{thm:CLT}]
Following Equation \eqref{eq.KTT} we have
\begin{eqnarray*}
     \frac{k(c,\leq\sigma)}{|K(c)|}  & = & \frac{s(c,\leq \sigma)}{2\,|K(c)|}  \,+\, \frac{s_p(c,\leq \sigma)}{2\,|K(c)|} 
     \\
     & = &    \frac{s(c,\leq \sigma)}{|T(c)|} \,\frac{|T(c)|}{2\,|K(c)|}  \,+\, \frac{s_p(c,\leq \sigma)}{2\,|K(c)|}  \,.
\end{eqnarray*}
It follows that 
\begin{eqnarray}
   \nonumber
   \left|   \frac{k(c,\leq\sigma)}{|K(c)|}  \,-\,   \frac{s(c,\leq \sigma)}{|T(c)|} \right|   & \leq &        \frac{s(c,\leq \sigma)}{|T(c)|} \,
    \left|  \frac{|T(c)|}{2\,|K(c)|}   \,-\,1 \right|   \,+\,   \frac{|T_p(c)|}{2\,|K(c)|}
    \\
    \label{eq.knears}
       & \leq &       \left|  \frac{|T(c)|}{2\,|K(c)|}   \,-\,1 \right|   \,+\,   \left| \frac{2\,|K(c)|-|T(c)|}{2\,|K(c)|} \right|\,.
 \end{eqnarray}
By Equation \eqref{eq.KTlim}, the right-hand side of Equation \eqref{eq.knears} converges to 0 as $c\rightarrow\infty$, uniformly in $\sigma$.   
Therefore Equation \eqref{eq.CLTk} follows from 
Equation \eqref{eq.CLT2m} on the Central Limit Theorem for the set $T(c)$, and the theorem is proved.
\end{proof}

\section{Some numerical calculations of variance}

We compute the variance for small crossing number $c$.  First for the even case wtih $c=2m$, we have that the variance is the same over both $T(2m)$ and $K(2m)$ because all knots in $T_p(2m)$ have signature 0 as in Definition \ref{def:palindrome}.  See Table \ref{tab:vareven}.

\begin{table}[!h]
\begin{tabular}{c|c}
$c=2m$ & Variance over $T(2m)$ and also $K(2m)$\\
\hline
6	& 1.6         \\
8	& 3.428571429 \\
10	& 5.364705882 \\
12	& 7.343108504 \\
14	& 9.336263736 \\
16	& 11.33626374 \\
18	& 13.33626374 \\
20	& 15.33626374
\end{tabular}
\caption{\label{tab:vareven} Some values for the variance over $T(2m)$ and $K(2m)$}
\end{table}

Observe that the decimal $0.33626374$ is approximately $153/455$ and that 455 is indeed a factor for $T(c)$ when $c=16,18,20$.

Next for the odd case wtih $c=2m+1$, we compute the variance for  only $T(2m+1)$  since the signatures of knots in $T_p(2m+1)$ are not as well-behaved.  See Table \ref{tab:varodd}.

\begin{table}[!h]
\begin{tabular}{c|c}
$c=2m+1$ & Variance over $T(2m+1)$\\
\hline
7	&	2.454545455 \\
9	&	4.348837209 \\
11	&	6.333333333 \\
13	&	8.332357247 \\
15	&	10.33663004 \\
17	&	12.33635531 \\
19	&	14.33628663
\end{tabular}
\caption{\label{tab:varodd} Some values for the variance over $T(2m+1)$}
\end{table}

\begin{conjecture}
The variance over $T(c)$ for crossing number $c$ is approximately $c-5+.336$.
\end{conjecture}

\bibliographystyle{amsalpha}

\newcommand{\etalchar}[1]{$^{#1}$}
\def\cprime{$'$}
\providecommand{\bysame}{\leavevmode\hbox to3em{\hrulefill}\thinspace}
\providecommand{\MR}{\relax\ifhmode\unskip\space\fi MR }
\providecommand{\MRhref}[2]{%
  \href{http://www.ams.org/mathscinet-getitem?mr=#1}{#2}
}
\providecommand{\href}[2]{#2}

\end{document}